\documentclass{amsart}

\usepackage{graphicx, graphics, amssymb, amsmath, amscd,latexsym}
\usepackage[abs]{overpic} %
\usepackage{color}
\usepackage[english]{babel}

\newtheorem{thm}{Theorem}[section]
\newtheorem{lem}[thm]{Lemma}
\newtheorem{conj}[thm]{Conjecture}
\newtheorem{obs}[thm]{Observation}
\newtheorem{cor}[thm]{Corollary}

\theoremstyle{definition}
\newtheorem{rem}[thm]{Remark}
\newtheorem{ex}[thm]{Example}
\newtheorem{que}[thm]{Question}

\def\Z{\mathbb Z}

\begin{document}
\dedicatory{Dedicated to Professors  Taizo Kanenobu,   Yasutaka Nakanishi, 
and Makoto~Sakuma 
  for their 60th birthday}
\title[]
{Fibered knots with the same $0$-surgery and the slice-ribbon conjecture}
\author{Tetsuya Abe and Keiji Tagami}
\address{Osaka City University Advanced Mathematical Institute, 
3-3-138 Sugimoto, Sumiyoshi-ku,
Osaka 558-8585 Japan}
\email{tabe@sci.osaka-cu.ac.jp}
\address{Department of Mathematics, Faculty of Science and Technology, Tokyo
University of Science, 2641 Yamazaki, Noda, Chiba, 278-8510, Japan}
\email{tagam\_keiji@ma.noda.tus.ac.jp}
\maketitle
\vskip -5mm
\vskip -5mm
\begin{abstract}
Akbulut and Kirby conjectured that two knots with the same $0$-surgery are concordant. In this paper, 
we prove that if the slice-ribbon conjecture is true,
then the modified Akbulut-Kirby's conjecture is false. 
We also give a  fibered potential counterexample to the slice-ribbon conjecture.
\end{abstract}

\section{Introduction}
The slice-ribbon conjecture 
 asks whether any slice knot in  $S^3$ bounds a ribbon disk in
the standard $4$-ball
 $B^4$  (see \cite{Fox}).
There are many studies on this conjecture
(cf.~\cite{{ATange2}, {ATange1}, {CD},  {GST},  {GJ},  {HKL}, {LM}, {Le}, {Li}}).
On the other hand, until recently, 
few direct consequences of the slice-ribbon conjecture were known. 
This situation has been changed by Baker.
He gave the following conjecture.
\par
\begin{conj}[{\cite[Conjecture 1]{Baker}}] \label{conj:baker1}
Let $K_0$ and $K_1$ be  fibered knots in $S^3$ supporting the tight contact structure.
If $K_0$ and $K_1$ are concordant, then $K_0=K_1$.
\end{conj}
\par
Baker proved a strong and direct consequence of the slice-ribbon conjecture
as follows:
\begin{thm}[{\cite[Corollary 4]{Baker}}] \label{thm:baker}
If the slice-ribbon conjecture is true, 
then Conjecture \ref{conj:baker1} is true. 
\end{thm}
Originally, Conjecture~\ref{conj:baker1} was motivated by Rudolph's old question \cite{Rudolph} which asks whether
the set of algebraic knots is linearly independent 
in the knot concordance group. 
Here we observe the following, which was implicit in \cite{Baker}. 
\par
\begin{obs}\label{obs:baker1}
If Conjecture~\ref{conj:baker1} is true, the set of prime fibered knots in $S^3$ supporting the tight contact structure is linearly independent in the knot concordance group (see Lemma~\ref{lem:baker1}). 
Moreover, the set of such knots contains algebraic knots (see Lemma~\ref{lem:baker2}).
In this sense, 
Conjecture~\ref{conj:baker1} is a generalization of Rudolph's question. 
Therefore Theorem~\ref{thm:baker} implies that if the slice-ribbon conjecture is true, then the set of algebraic knots is linearly independent in the knot concordance group --an affirmative answer of Rudolph's question--.
\end{obs}
%
%
%
%
Theorem \ref{thm:baker} and Observation \ref{obs:baker1}
make the slice-ribbon conjecture more important and fascinating.
\par
In this paper, we give another consequence of the slice-ribbon conjecture.
To state our main result,
we recall Akbulut and Kirby's conjecture on knot concordance
in the Kirby's problem list \cite{Kirby}.
\par
\begin{conj}
[{\cite[Problem 1.19]{Kirby}}]\label{conj:AK}
If $0$-surgeries on two knots give the same 3-manifold, then the knots 
are concordant. 
\end{conj}
Livingston \cite{Livingston} demonstrated a knot $K$ such that it is not concordant to 
$K^{r}$, where $K^{r}$ is the knot obtained from $K$ by reversing the orientation.
Therefore Conjecture~\ref{conj:AK} is false since 
$0$-surgeries on $K$ and $K^{r}$ give the same 3-manifold,
however, the following conjecture seems to be still open.
\begin{conj}
\label{conj:AK2}
If $0$-surgeries on two knots give the same 3-manifold, 
then the knots with relevant orientations are concordant.
\end{conj}
Note that Cochran, Franklin, Hedden and Horn \cite{CFHH} obtained a closely related result
to Conjecture \ref{conj:AK2}.
Indeed they gave a negative answer to the following question:
``If $0$-surgeries on two knots are integral homology cobordant,
preserving the homology class of the positive meridians, are the knots concordant?"
\par
Our main result is the following.
\begin{thm}\label{thm:main}
If the slice-ribbon conjecture is true,
then Conjecture~\ref{conj:AK2} is false\footnote{
Recently, Kouichi Yasui \cite{Yasui} proved that there are infinitely many
counterexamples of Conjecture \ref{conj:AK2}.}.
\end{thm}
In Section \ref{sec:main},
we will prove Theorem \ref{thm:main}.
We  outline the proof as follows: Let $K_0$ and $K_1$ be the 
unoriented knots depicted in Figure~\ref{fig;knots.eps}
and give arbitrary orientations on $K_0$ and $K_1$.

\begin{figure}[!htb]
\centering
\begin{overpic}[width=.8\textwidth]{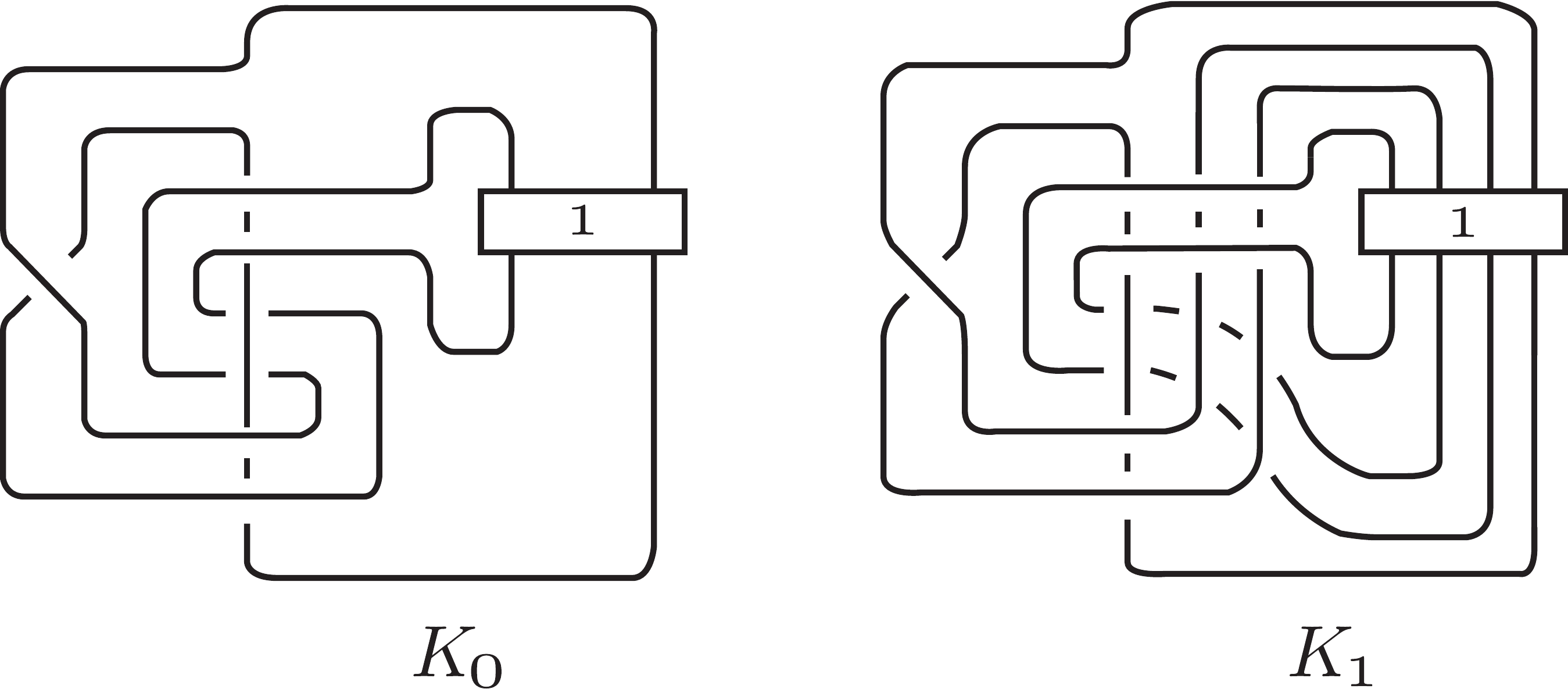}
\end{overpic}
\caption{The definitions of $K_0$ and $K_1$.
Each rectangle labeled~$1$ implies a full twist.}
\label{fig;knots.eps}
\end{figure}

By using annular twisting techniques developed in \cite{AJOT, AJLO, Osoinach, Teragaito}, 
we see that 
$0$-surgeries on $K_0$ and $K_1$ give the same 3-manifold.
On the other hand,
by  Miyazaki's result  \cite{Miyazaki}, 
we can prove that $K_{0}\#\overline{K_1}$ is not ribbon,
where  $K_{0}\#\overline{K_1}$ denotes the connected sum of $K_0$ and the mirror image of $K_1$.
Suppose that the slice-ribbon conjecture is true.
Then  
$K_{0}\#\overline{K_1}$ is not slice.
Equivalently,  $K_0$ and $K_1$ are not concordant.
As a summary,
$0$-surgeries on $K_0$ and $K_1$ give the same 3-manifold,
however,
they are not concordant if the slice-ribbon conjecture is true,
implying that  Conjecture~\ref{conj:AK2} is false.
\par
Here we consider the following question.
\begin{que}
Are the knots $K_0$ and $K_1$ in Figure~\ref{fig;knots.eps} concordant?
\end{que}

\noindent
This question is interesting since the proof of Theorem \ref{thm:main} tells us 
the following:
\begin{enumerate}
\item If $K_0$ and $K_1$ are concordant, then $K_0 \# \overline{K}_1$ is a counterexample 
to the slice-ribbon conjecture since $K_0 \# \overline{K}_1$ is not ribbon. 
\item If $K_0$ and $K_1$ are not concordant, then
Conjecture \ref{conj:AK2} is false since 
$0$-surgeries on $K_0$ and $K_1$ give the same 3-manifold. 
\end{enumerate}
\par
%
This paper is organized as follows: 
In Section~\ref{sec:main}, 
we prove Theorem~\ref{thm:main}. 
In section~\ref{sec:Baker}, 
we consider consequences of Baker's result in \cite{Baker}.
In Appendix A, 
we give a short review for Miyazaki's 
in-depth study 
on ribbon fibered knots which is based on the theorem of Casson and Gordon \cite{CG}.
In Appendix B, 
we recall twistings, annulus twists and annulus presentations. 
Moreover, we define annulus presentations compatible with fiber surfaces and 
study a relation between annulus twists and fiberness of knots. 
Finally, 
we describe monodromies of the fibered knots obtained 
from $6_3$ (with an annulus presentation)
by annulus twists.
Using these  monodromies, 
we can distinguish these knots. 
\par
\ 
\par
\noindent
\text{\textbf{Notations}.}
Throughout this paper, 
we will work in the smooth category.
Unless otherwise stated, we suppose that all knots are oriented.
Let $K$ be a knot in  $S^3$.
We denote $M_{K}(0)$ the $3$-manifold obtained from $S^{3}$ by $0$-surgery on $K$ in $S^3$,
and by $[K]$
the concordance class of $K$.
For an oriented compact surface $F$ 
with a single boundary component and a diffeomorphism $f\colon F\rightarrow F$ fixing the boundary,
we denote by $\widehat{F}$ the closed surface $F\cup D^{2}$ and by $\widehat{f}$ 
the extension $f\cup \operatorname{id}\colon \widehat{F}\rightarrow \widehat{F}$. 
We denote by $t_{C}$ the right-handed Dehn twist along a simple closed curve $C$ on $F$. 
%
\section{Proof of Theorem~\ref{thm:main}}\label{sec:main}
In this section, we prove our main theorem. 
The main tools are Miyazaki's result \cite[Theorem~5.5]{Miyazaki} and annular twisting techniques developed in 
\cite{AJOT, AJLO, Osoinach, Teragaito}. 
For the sake of completeness,
we will review these results in Appendices A and B. 
%
\begin{proof}[Proof of Theorem~\ref{thm:main}]
Let $K_0$ and $K_1$ be the unoriented knots as in Figure~\ref{fig;knots.eps}
and  give arbitrary orientations on $K_0$ and $K_1$.
By \cite[Theorem 2.3]{Osoinach} (see Lemma~\ref{lem:0-surgery}), 
$0$-surgeries on $K_0$ and $K_1$ give the same 3-manifold 
(for detail, see Appendix B.
 In this case, $K_{0}$ admits an annulus presentation as in Figure~\ref{figure:annulus-pre} and $K_{1}=A(K_{0})$).
\par
On the other hand, 
$K_0 \# \overline{K}_1$ 
is not a ribbon knot as follows:
First, note that $K_0$ is the fibered knot $6_3$ in Rolfsen's knot table,
see KnotInfo \cite{CL}.
By Gabai's theorem in \cite{Gabai}, 
$K_1$ is also fibered since $0$-surgeries on $K_1$ and $K_2$ give the same 3-manifold (see also Remark~\ref{rem:gabai}). 
Therefore $K_0 \# \overline{K}_1$ is a fibered knot. 
Here we can see that $K_0$ and $K_1$ are different knots 
(for example, by calculating the Jones polynomials of $K_0$ and $K_1$).
Also, we see that $K_{0}$ and $K_{1}$ have the same irreducible Alexander polynomial 
\[ \Delta _{K_{0}}(t)=\Delta _{K_{1}}(t)=	1-3t+ 5t^{2}-3t^{3}+ t^{4}.\] 
By Miyazaki's result \cite[Theorem~5.5]{Miyazaki} (or Corollary~\ref{cor:Miyazaki}), the knot 
$K_0 \# \overline{K}_1$ is not ribbon. 

Suppose that the slice-ribbon conjecture is true.
Then  
$K_{0}\#\overline{K_1}$ is not slice.
Equivalently,  $K_0$ and $K_1$ are not concordant.
Therefore,  if the slice-ribbon conjecture is true,
then Conjecture~\ref{conj:AK2} is false.
\end{proof}
\section{Observations on Baker's result} \label{sec:Baker}
In this section, we consider consequences of Baker's result in \cite{Baker}. 
\par
First, we recall some definitions. 
A fibered knot in $S^3$ is called \textit{tight} if it supports the tight contact structure (see \cite{Baker}). 
A set of knots is \textit{linearly independent} in the knot concordance group if it is linearly independent in the knot concordance group
as a $\Z$-module. 
We observe the following. 
\par
\begin{lem} \label{lem:baker1}
If Conjecture \ref{conj:baker1} is true,
then the set of prime tight fibered knots in $S^3$
is linearly independent in the knot concordance group.
\end{lem}
\begin{proof}
Let $K_1, K_2, \ldots, K_n$ be distinct prime tight fibered knots. 
Suppose that for some integers $a_{1}, \ldots, a_{n}$ we have 
\[ a_1 [K_1] + \cdots + a_n [K_n]=0.\]
We will prove that 
if Conjecture \ref{conj:baker1} is true, then
$a_1=a_2= \cdots= a_n=0$.
When 
$a_1 \ge 0$, $\ldots $, $a_n \ge 0$, then
\[ [ (\#^ {a_1} K_1) \#\cdots \# (\#^ {a_n} K_n) ]
=0. \]
Note that $(\#^ {a_1} K_1) \#\cdots \# (\#^ {a_n} K_n)$ 
is a tight fibered knot. 
If Conjecture \ref{conj:baker1} is true, then 
\[   (\#^ {a_1} K_1)    \#\cdots  \# (\#^ {a_n} K_n)  \] is the unknot.
By the prime decomposition theorem of knots, we obtain 
\[ a_1=a_2= \cdots= a_n=0. \]
When 
$a_1 \le 0$, $\ldots $, $a_n \le 0$, then
\[ [  (\#^ {-a_1} K_1) \#\cdots  \# (\#^ {-a_n} K_n) ]
= 0. \] 
By the same argument, 
we obtain 
\[ a_1=a_2= \cdots= a_n=0. \]
For the other case, 
we may assume that 
$a_1 \ge 0$,  $\ldots $, $a_m \ge 0$
and $a_{m+1}  \le 0$,  $\ldots $, $a_n \le 0$
by changing the order of the knots. 
Then we obtain 
\[ a_1 [K_1]  +\cdots + a_m [K_m]
=  (-a_{m+1}) [K_{m+1}] +\cdots + (-a_n) [K_n].\]
Equivalently, 
\[ [  (\#^ {a_1} K_1)    \#\cdots  \# (\#^ {a_m} K_m) ]
= [(\#^ {-a_{m+1}} K_{m+1})    \#\cdots  \# (\#^ {-a_n} K_n)]. \]
Note that  $(\#^ {a_1} K_1)  \#\cdots \# (\#^ {a_m} K_m)$ and 
$(\#^ {-a_{m+1}} K_{m+1})  \#\cdots \# (\#^ {-a_n} K_n)$
are tight fibered knots.
If Conjecture \ref{conj:baker1}  is true,
then 
\[   (\#^ {a_1} K_1)    \#\cdots  \# (\#^ {a_m} K_m) =
(\#^ {-a_{m+1}} K_{m+1})    \#\cdots  \# (\#^ {-a_n} K_n). \]
By the prime decomposition theorem of knots, we obtain
\[ a_1=a_2= \cdots= a_n=0. \vspace{-2em} \]
\end{proof}
Lemma~\ref{lem:baker1} leads us to ask which knots are (prime) tight fibered. 
Recall that \textit{algebraic knots} are links of isolated singularities of complex curves and \textit{$L$-space knots} are those admitting positive Dehn
surgeries to  L-spaces\footnote{
In our definition, the left-handed trefoil is not an $L$-space knot.
Note that some authors define \textit{$L$-space knots} to be those admitting non-trivial  Dehn
surgeries to  L-spaces.
In this definition, the left-handed trefoil is  an $L$-space knot.}. 
\begin{lem} \label{lem:baker2}
We have the following.
\begin{enumerate}
\item A fibered knot is tight if and only if it is strongly quasipositive. \label{quasipositive}
\item An algebraic knot is a prime tight fibered knot. \label{algebra}
\item An $L$-space knot is a prime tight fibered knot. \label{L-space}
\item A divide knot is a tight fibered knot. \label{divide}
\item A positive fibered knot is a tight fibered knot. \label{positive}
\item An almost positive fibered knot is a tight fibered knot. \label{almost-positive}
\end{enumerate}
\end{lem}
\begin{proof}
(\ref{quasipositive})
This follows from  \cite[Proposition 2.1]{Hedden} (see also \cite{BI}). 
\par
(\ref{algebra}) It is well known that any algebraic knot is fibered and strongly quasipositive.
By (\ref{quasipositive}), it is tight fibered.
In fact,   any algebraic knot is an iterated cable of a torus knot.
This implies that  it is prime.
For the details on algebraic knots,
see \cite{EN}, \cite{Hillman}, \cite{Wall}.

(\ref{L-space}) By  \cite{Ni1, Ni2} (see also \cite{Ghiggini}, \cite{OZ}), 
an $L$-space knot is fibered.
Hedden \cite{Hedden} proved that it is tight.
It is also known that an $L$-space knot is prime,
see \cite{Krcatovich}.

(\ref{divide}) A'Campo \cite{Acampo2} 
proved that  a divide knot is  fibered
 and its monodromy is a product of positive Dehn twists.
Such a fibered knot is known to be tight,
for example, see Remark 6.5 in \cite{GK}.
For the details, see \cite{book-contact}.

(\ref{positive}) Nakamura \cite{Nakamura} and Rudolph \cite{Rudolph2} proved that any positive knot
is strongly quasipositive.
By (\ref{quasipositive}), a positive fibered knot is tight.

(\ref{almost-positive}) The authors  \cite{ATagami} proved that any  almost positive fibered knot
is strongly quasipositive.
By (\ref{quasipositive}), an almost positive fibered knot is tight.
\end{proof}
\begin{rem}
By  Theorem \ref{thm:baker} and Lemmas~\ref{lem:baker1} and \ref{lem:baker2},
if the slice-ribbon conjecture is true, then the set of L-space knots in $S^3$ is also linearly independent in the knot concordance group. 
\end{rem}
The following conjecture may be manageable  than Conjecture \ref{conj:baker1}.
\begin{conj}
The set of L-space knots  in $S^3$ is linearly independent in the knot concordance group.
In particular, if L-space knots $K_0$ and $K_1$ are concordant, then $K_0=K_1$. 
\end{conj}
%
%
%
%
%
%
%
%
\section{Appendix A: Miyazaki's results on ribbon fibered  knots}\label{sec:miyazaki}
In this appendix, we recall  Miyazaki's results \cite{Miyazaki} on non-simple ribbon fibered knots, in particular, on composite ribbon fibered knots. 
\par
Let $K_{i}$ be a knot in a homology $3$-sphere $M_{i}$ for $i=0,1$. 
We write 
\begin{center}$(M_{1}, K_{1})\geq (M_{0}, K_{0})$ (or simply $K_{1}\geq K_{0}$)
\end{center}
if there exist a $4$-manifold $X$ with $H_{*}(X, \Z) \simeq H_{*}(S^3 \times I, \Z)$ and an annulus $A$ embedded into $X$ such that 
\begin{align*}
(\partial X, A\cap \partial X)&=(M_{1}, K_{1})\sqcup (M^{r}_{0}, K^{r}_{0}),\\
\pi_{1}(M_{1}\setminus K_{1})&\rightarrow \pi_{1}(X\setminus A) \text{is surjective, and}\\
\pi_{1}(M_{0}\setminus K_{0})&\rightarrow \pi_{1}(X\setminus A) \text{is injective}, 
\end{align*}
where $(M^{r}_{0}, K^{r}_{0})$ is $(M_{0}, K_{0})$ with reversed orientation. 
We say $K_{1}$ is {\it homotopically ribbon concordant} to $K_{0}$ if $K_{1}\geq K_{0}$. 
Note that this is a generalization of the notion of ``ribbon concordant'', 
see in \cite[Lemma 3.4]{Gordon}. 
A knot $K$ in a homotopy $3$-sphere $M$ is {\it homotopically ribbon} if $K\geq U$, where $U$ is the unknot in $S^{3}$. 
A typical example of a homotopically ribbon knot is a ribbon knot in $S^{3}$ (for detail see \cite[p.3]{Miyazaki}). 
\par
%
\begin{thm}[{\cite[Theorem~5.5]{Miyazaki}}]
For $i=1, \ldots, n$, let $K_{i}$ be a prime fibered knot in a homotopy $3$-sphere $M_{i}$ satisfying one of the following:  
\begin{itemize}
\item $K_{i}$ is minimal with respect to $`` \geq $'' among all fibered knots in homology spheres, 
\item there is no $f(t)\in\mathbf{Z}[t]\setminus\{\pm t^{k}\}_{k}$ such that $f(t)f(t^{-1})|\Delta _{K_{i}}(t)$.  
\end{itemize}
If $K_{1}\#\cdots\#K_{n}$ is homotopically ribbon, then the set 
$\{1, \ldots, n\}$ can be paired into $\sqcup_{s=1}^{m}\{i_{s}, j_{s}\}$ such that $K_{i_{s}}= \overline{K_{j_{s}}}$, 
where $\overline{K_{j_{s}}}$ is $(M^{r}_{j_{s}}, K^{r}_{j_{s}})$. 
\end{thm} 
\begin{rem}
By a solution of the geometrization conjecture (see \cite{Perelman1}, 
\cite{Perelman2} and \cite{Perelman3}), each homotopy $3$-sphere $M_{i}$ is $S^{3}$ in the above theorem. 
\end{rem}
As a corollary, we obtain the following. 
\begin{cor}\label{cor:Miyazaki}
Let $K_{0}$ and $K_{1}$ be fibered knots in $S^{3}$ with irreducible Alexander polynomials. 
If $K_{0}\#\overline{K_{1}}$ is ribbon,
then $K_{0}=K_{1}$. 
\end{cor}
\section{Appendix B: Twistings, Annulus twists and Annulus presentations}\label{sec:twist}
In this appendix, 
we recall two operations. 
One is {\it twisting}, and the other is  {\it annulus twist} \cite{AJOT}.  
In a certain case,
annulus twists are expressed in terms of twistings
and preserve  some  properties  of knots. 
Finally, 
we describe  monodromies of the fibered knots obtained 
from $6_3$ (with an annulus presentation) by annulus twists.
We begin this appendix with recalling the definition of an open book decomposition of a $3$-manifold.
\subsection{Open book decompositions}\
Let $F$ be an oriented surface with boundary and $f\colon F\rightarrow F$  a diffeomorphism on $F$ fixing the boundary. 
Consider the pinched mapping torus 
\[\widehat{M}_{f}=F\times [0,1]/_{\sim},\]
where the equivalent relation $\sim$ is defined as follows:
\begin{enumerate}
\item  $(x,1)\sim (f(x),0)$ for $x \in F$, and 
\item  $(x,t)\sim (x,t')$ for $x \in \partial F$ and $t$, $t'$ $\in [0,1]$. 
\end{enumerate}
Here, we orient $[0,1]$ from $0$ to $1$ and we give an orientation of $\widehat{M}_{f}$ by the orientations of $F$ and $[0,1]$.
Let $M$ be a closed  oriented $3$-manifold.
If there exists an orientation-preserving  diffeomorphism from $\widehat{M}_{f}$ to $M$, 
the pair $(F, f)$ is called an \textit{open book decomposition} of $M$. 
The map $f$ is called the \textit{monodromy} of $(F, f)$. 
Note that we can regard $F$ as a surface in $M$.
The boundary of $F$ in $M$, denoted by $L$,
is called a \textit{fibered link} in $M$, and $F$ is called a \textit{fiber surface} of $L$. 
The \textit{monodromy of $L$} is defined by the monodromy $f$ of the open book decomposition $(F, f)$. 
\subsection{Twistings and annulus twists}\
Let $M$ be a closed  oriented $3$-manifold, and $(F, f)$  an open book decomposition of  $M$.
Let $C$ be a simple closed curve on  a fiber surface  $F \subset M$.
Then,  {\it a twisting along $C$ of order $n$} is defined as performing $(1/n)$-surgery along $C$ with respect to the framing determined by $F$. 
Then we obtain the following.
\begin{lem}[Stallings]\label{lem:Stallings}
The resulting manifold obtained from $M$ by a twisting along $C$ of order $n$
is (orientation-preservingly) diffeomorphic to $\widehat{M}_{t_{C}^{-n}\circ f}$.
\end{lem}
For a proof of this  lemma,
see Figure~\ref{figure:mapping_torus} (see also  \cite{Bonahon}, \cite{book-contact} and \cite{Stallings}).
Lemma \ref{lem:Stallings} implies that, by a  twisting along $C$ of order $n$,
the fibered link with monodromy $f$ 
is changed into the fibered link with monodromy $t_{C}^{-n}\circ f$. 
\begin{figure}[htbp]
\centering
\includegraphics[scale=0.58]{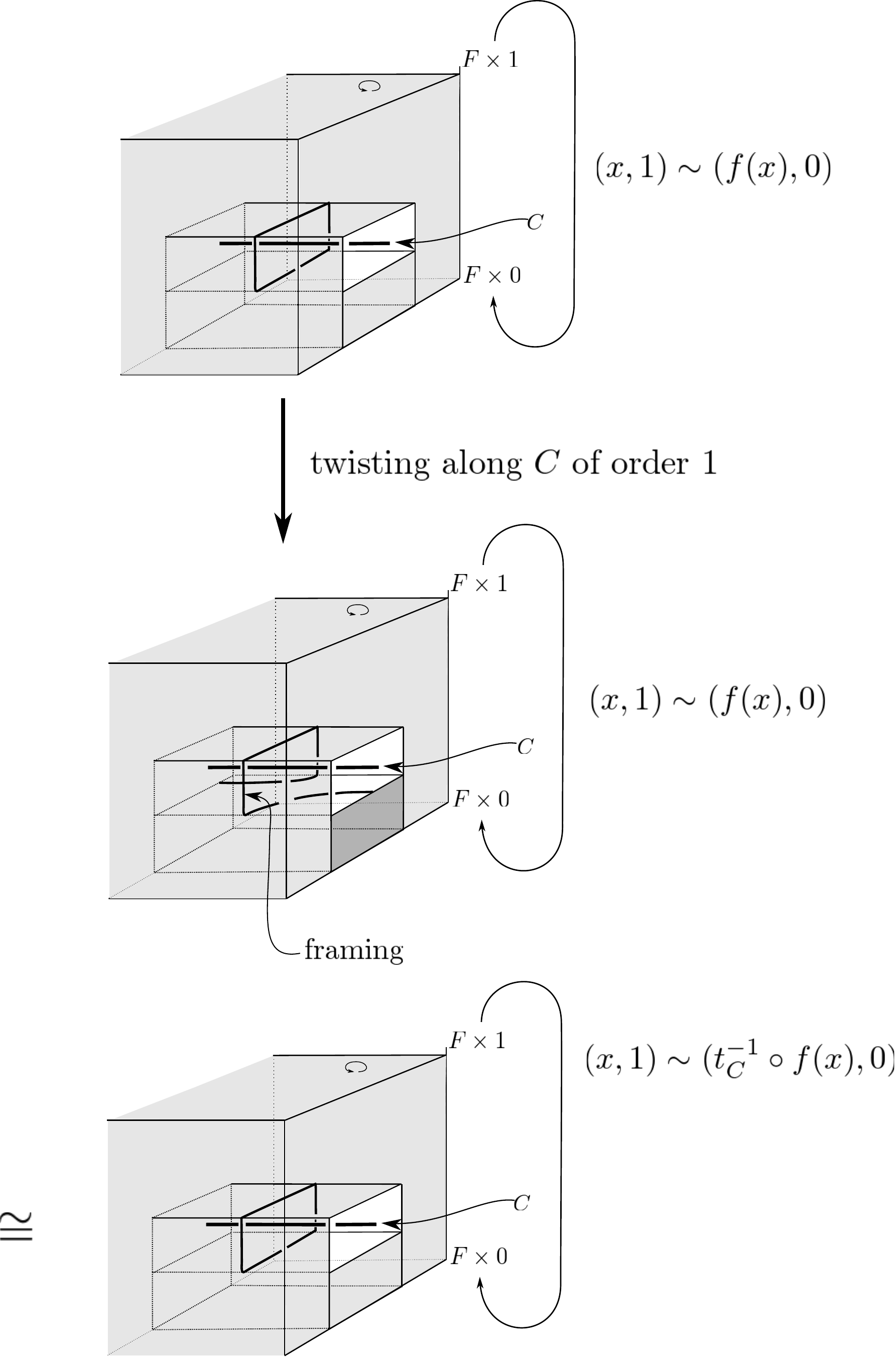}
\caption{
The top picture is $\widehat{M}_{f}$,
the middle picture is the resulting manifold obtained from $\widehat{M}_{f}$ by a twisting along $C$ of order $1$ and the bottom picture is $\widehat{M}_{t_{C}^{-1}\circ f}$. 
In the pictures, we remove a tubular neighborhood of $C$. 
The last diffeomorphism is given by twisting the deep gray area
(which is the solid torus below the neighborhood of $C$)   in the middle picture to the left. 
}
\label{figure:mapping_torus}
\end{figure}
\begin{rem}
Our definition on the pinched mapping torus differs from Bonahon's \cite{Bonahon}. 
We glue $(x,1)$ and $(f(x), 0)$ in the pinched mapping torus. 
On the other hand, $(x, 0)$ and $(f(x), 1)$ are glued in Bonahon's paper. 
\end{rem}
\par
%
%
%
Hereafter we only deal with the $3$-sphere $S^3$.
Let $A\subset S^{3}$ be an embedded annulus and $\partial A=c_{1}\cup c_{2}$. 
Note that $A$ may be knotted and twisted. 
In Figure~\ref{figure:annulus}, 
we draw an unknotted  and twisted annulus. 
An {\it $n$-fold annulus twist along $A$} is to apply $(+1/n)$-surgery along $c_{1}$ and $(-1/n)$-surgery along $c_{2}$ with respect to the framing determined by the annulus $A$. For simplicity, we call a $1$-fold annulus twist along $A$ 
an \textit{annulus twist along $A$}.
\begin{rem}
An $n$-fold annulus twist does not change the ambient $3$-manifold $S^{3}$,
(see \cite[Lemma 2.1]{ATange1} or \cite[Theorem 2.1]{Osoinach}).
However, each surgery along $c_{1}$ (resp. surgery along $c_{2}$) 
often changes the ambient $3$-manifold $S^{3}$.
For example, if $A$ is an unknotted annulus with $k$ full-twists,
the $n$-fold annulus twist along $A$ is to apply $(k+1/n)$-surgery along $c_{1}$ and $(k-1/n)$-surgery along $c_{2}$
with respect to the Seifert framings. Therefore 
 each surgery along $c_{1}$  (resp.~surgery along $c_{2}$) 
indeed 
changes the ambient $3$-manifold $S^{3}$ frequently.
\end{rem}
\begin{figure}[htbp]
\centering
\includegraphics[scale=0.4]{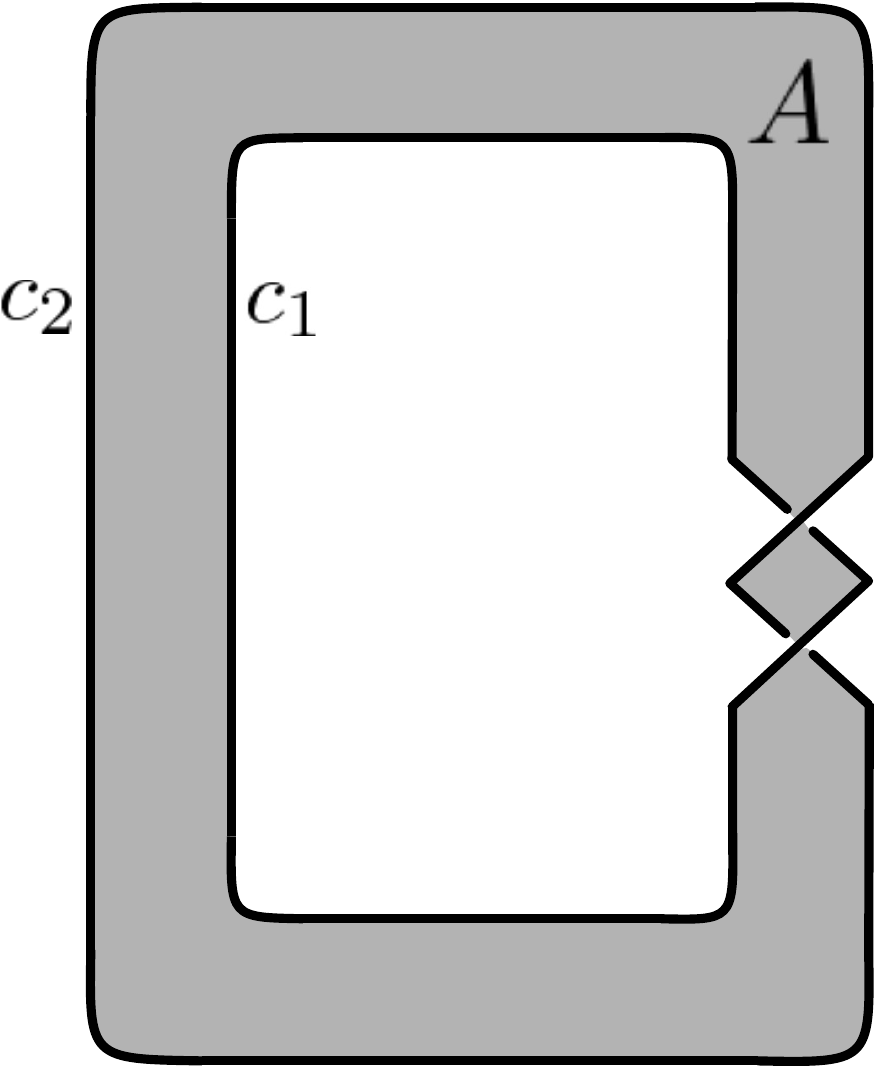}
\caption{An unknotted annulus $A\subset S^{3}$ with a $+1$ full-twist. }\label{figure:annulus}
\end{figure}
\subsection{Annulus presentations}\label{subsec:annulus-pre}\ 
The first author, Jong, Omae and Takeuchi \cite{AJOT}
 introduced the notion of an annulus presentation of a knot 
(in their paper it is called ``band presentation"). 
Here, we extend the definition of annulus presentations of knots. \par
Let $A\subset S^{3}$ be an embedded annulus with $\partial A=c_{1}\cup c_{2}$,
which may be knotted and twisted. 
Take an embedding of a band $b\colon I\times I\rightarrow S^{3}$ such that 
\begin{itemize}
\item $b(I\times I)\cap \partial A=b(\partial I\times I)$, 
\item $b(I\times I)\cap \operatorname{int} A$ consists of ribbon singularities, and 
\item $A\cup b(I\times I)$ is an immersion of an orientable surface, 
\end{itemize}
where $I=[0, 1]$. 
If a knot $K$ is isotopic to the knot $(\partial A\setminus b(\partial I\times I))\cup b(I\times \partial I)$, 
then we say that $K$ admits an {\it annulus presentation} $(A, b)$.
\begin{figure}[htbp]
\centering
\includegraphics[scale=0.8]{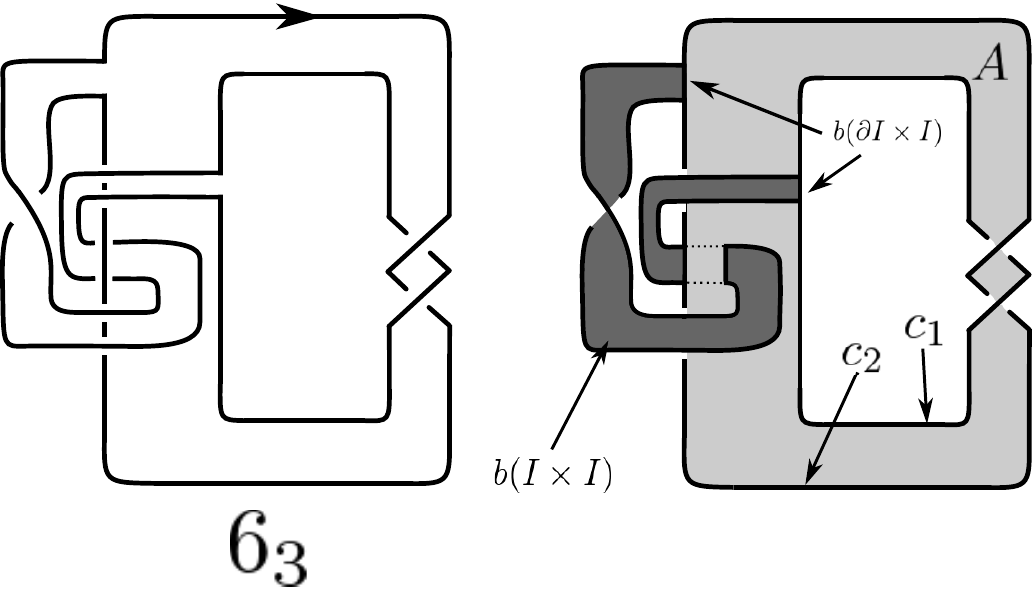}
\caption{The definitions of the knot $6_{3}$ (left)  and its annulus presentation (right). 
}\label{figure:annulus-pre}
\end{figure}
\begin{ex}
The knot $6_{3}$ (with an arbitrary orientation) 
 admits an annulus presentation $(A, b)$, see  Figure \ref{figure:annulus-pre}. 
\end{ex}
Let $K$ be a knot admitting an annulus presentation $(A, b)$. 
Then, by $A^{n}(K)$, 
we denote the knot obtained from $K$ by  $n$-fold annulus twist along $\widetilde{A}$ with 
$\partial\widetilde{A}=\widetilde{c}_{1}\cup \widetilde{c}_{2}$, where $\widetilde{A}\subset A$ is a shrunken annulus. 
Namely,  
$\overline{A\setminus \widetilde{A}}$ is a disjoint union of two annuli, 
each $\widetilde{c}_{i}$ is isotopic to $c_{i}$ in $\overline{A\setminus \widetilde{A}}$ for $i=1,2$ and 
$A\setminus (\partial A\cup \widetilde{A})$ does not intersect $b(I\times I)$. 
\begin{figure}[h]
\centering
\includegraphics[scale=0.8]{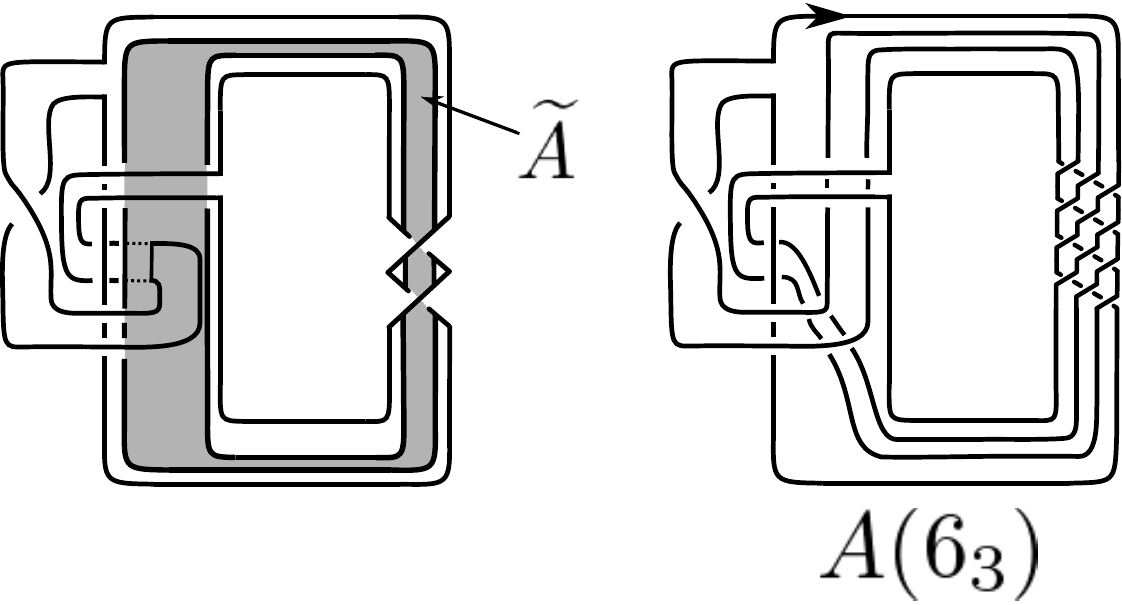}
\caption{A shrunken annulus $\tilde{A}$ for the annulus presentation of $6_{3}$ (left)  and
the knot  $A(6_{3})$ (right). 
}\label{figure:annulus-twist}
\end{figure}
For simplicity, we denote  $A^{1}(K)$ by $A(K)$.
\begin{ex}
We consider the knot $6_{3}$
with the annulus presentation $(A,b)$ in  Figure \ref{figure:annulus-pre}.
Then $A(6_{3})$ is the right picture in Figure \ref{figure:annulus-twist}.
\end{ex}
The following important lemma is a special case of Osoinach's result   \cite[Theorem 2.3]{Osoinach}. 
\begin{lem}\label{lem:0-surgery} 
Let $K$ be a knot admitting an annulus presentation $(A, b)$. 
Then, the $3$-manifold $M_{A^{n}(K)}(0)$ does not depend on $n\in\mathbf{Z}$. 
\end{lem}
%
%
\subsection{Compatible annulus presentations and twistings}\label{subsec:com-anuu}\
Let $K\subset S^{3}$ be a fibered knot admitting an annulus presentation $(A, b)$,
and  $F$  a fiber surface of $K$. 
We say that $(A, b)$ is {\it compatible with  $F$} if there exist simple closed curves $c'_{1}$ and $c'_{2}$ on $F$ such that 
\begin{itemize}
\item 
$\partial \widetilde{A}=\widetilde{c}_{1}\cup \widetilde{c}_{2}$ is isotopic to $c'_{1}\cup c'_{2}$ in $S^{3}\setminus K$, where $\widetilde{A}\subset A$ is a shrunken annulus defined in Section~\ref{subsec:annulus-pre}, and 
\item each annular neighborhood of $c'_{i}$ in $F$ ($i=1,2$) is isotopic to $A$ in $S^{3}$.
\end{itemize}
Let $\widetilde{c}_{1}\cup \widetilde{c}_{2}$ be the framed link with framing $(1/n, -1/n)$ with respect to the framing determined by the annulus $A$,
and  $c'_{1}\cup c'_{2}$  the framed link with framing $(1/n, -1/n)$ with respect to the framing determined by the fiber surface $F$. 
Then, by the first compatible condition, 
$\widetilde{c}_{1}\cup \widetilde{c}_{2}$ is equal to $c'_{1}\cup c'_{2}$ as links in $S^{3}\setminus K$. 
Moreover, by the second compatible condition, their framings coincide. 
As a result, $\widetilde{c}_{1}\cup \widetilde{c}_{2}$ is equal to $c'_{1}\cup c'_{2}$ as framed links in $S^{3}\setminus K$. 
Hence, if $K$ is a fibered knot with $(A, b)$ which is compatible with the fiber surface $F$,
 then $A^{n}(K)$ is the knot obtained from $K$ by twistings along $c'_{1}$ and $c'_{2}$ of order $+n$ and $-n$, respectively. 
In particular, 
$A^{n}(K)$ is a fibered knot and the monodromy of $A^{n}(K)$ is 
$t_{c'_{1}}^{-n}\circ t_{c'_{2}}^{n}\circ f$, 
where $f$ is the  monodromy of $K$.
As a summary,
we obtain the following. 
\begin{lem}\label{lem:annulus_fiber}
Let $K\subset S^{3}$ be a fibered knot admitting a compatible annulus presentation $(A, b)$. 
Then $A^{n}(K)$ is also fibered for any $n\in \mathbf{Z}$. 
Moreover, the monodromy of $A^{n}(K)$ is 
\[
t_{c'_{1}}^{-n}\circ t_{c'_{2}}^{n}\circ f, 
\]
where $f$ is the monodromy of $K$, and $c'_{1}$ and $c'_{2}$ are simple closed curves which give the compatibility of $(A, b)$. 
\end{lem}

\begin{figure}[tbp]
\centering
\includegraphics[scale=0.75]{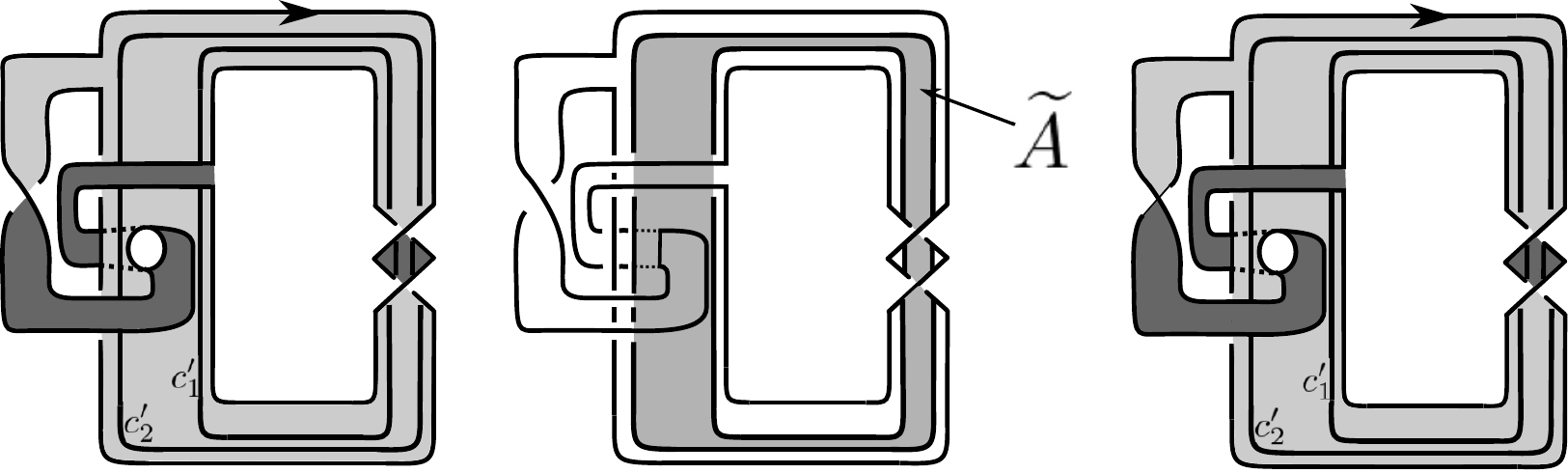}
\caption{A fiber surface $F$ of $6_3$ (left) and a shrunken annulus $\widetilde{A}$ (center).
The annulus presentation $(A,b)$ of  $6_{3}$ is compatible with the fiber surface $F$ (right). 
}\label{figure:annulus-presentation}
\end{figure}

\begin{ex}
We consider the knot $6_3$ with the annulus presentation $(A, b)$  in  Figure \ref{figure:annulus-pre}.
It is known that $6_3$ is fibered.
We choose a fiber surface 
as in the left picture in Figure~\ref{figure:annulus-presentation},
and denote it by $F$.
In this case,  the annulus presentation $(A, b)$ is compatible with $F$.
Indeed we define simple closed curves $c'_{1}$ and $c'_{2}$ on $F$ 
by  $\widetilde{c}_{1}$ and $\widetilde{c}_{2}$,
where  $\partial \widetilde{A}=\widetilde{c}_{1}\cup \widetilde{c}_{2}$.
Then $c'_{1}\cup c'_{2}$ clearly satisfies the compatible conditions.
\end{ex}

\begin{rem}\label{rem:gabai}
Let $K_{1}$ and $K_{2}$ be knots which have the same $0$-surgery.
Gabai \cite{Gabai} proved that if $K_{1}$ is fibered, then $K_{2}$ is also fibered. 
Here let $K$ be a fibered knot admitting an annulus presentation $(A, b)$ (which may not be compatible with the fiber surface for $K$). Then, by Lemma~$\ref{lem:0-surgery}$ and the above fact (Gabai's theorem),
$A^{n}(K)$ is also fibered. 
\end{rem}
\subsection{The monodromy of $A^{n}(6_{3})$}\
At first, we describe the monodromy of $6_{3}$.
We draw a fiber surface of $6_{3}$ as a plumbing of some Hopf bands (see Figure~\ref{figure:monodromy}). 
From  Figures~\ref{figure:monodromy}  and \ref{figure:monodromy3},
the monodromy of $6_{3}$ is given by $ t^{-1}_{d}\circ t_{b}\circ t^{-1}_{c}\circ t_{a}$. 
\par
Now  we describe the monodromy of $A^{n}(6_{3})$. 
From Figures~\ref{figure:monodromy2}, \ref{figure:monodromy3}, and 
Lemma~\ref{lem:annulus_fiber}, the monodromy $f_{n}$ of $A^{n}(6_{3})$ is given by $t_{c'_{1}}^{-n}\circ t_{c'_{2}}^{n}\circ t^{-1}_{d}\circ t_{b}\circ t^{-1}_{c}\circ t_{a}$. 
\begin{figure}[h]
\centering
\includegraphics[scale=0.9]{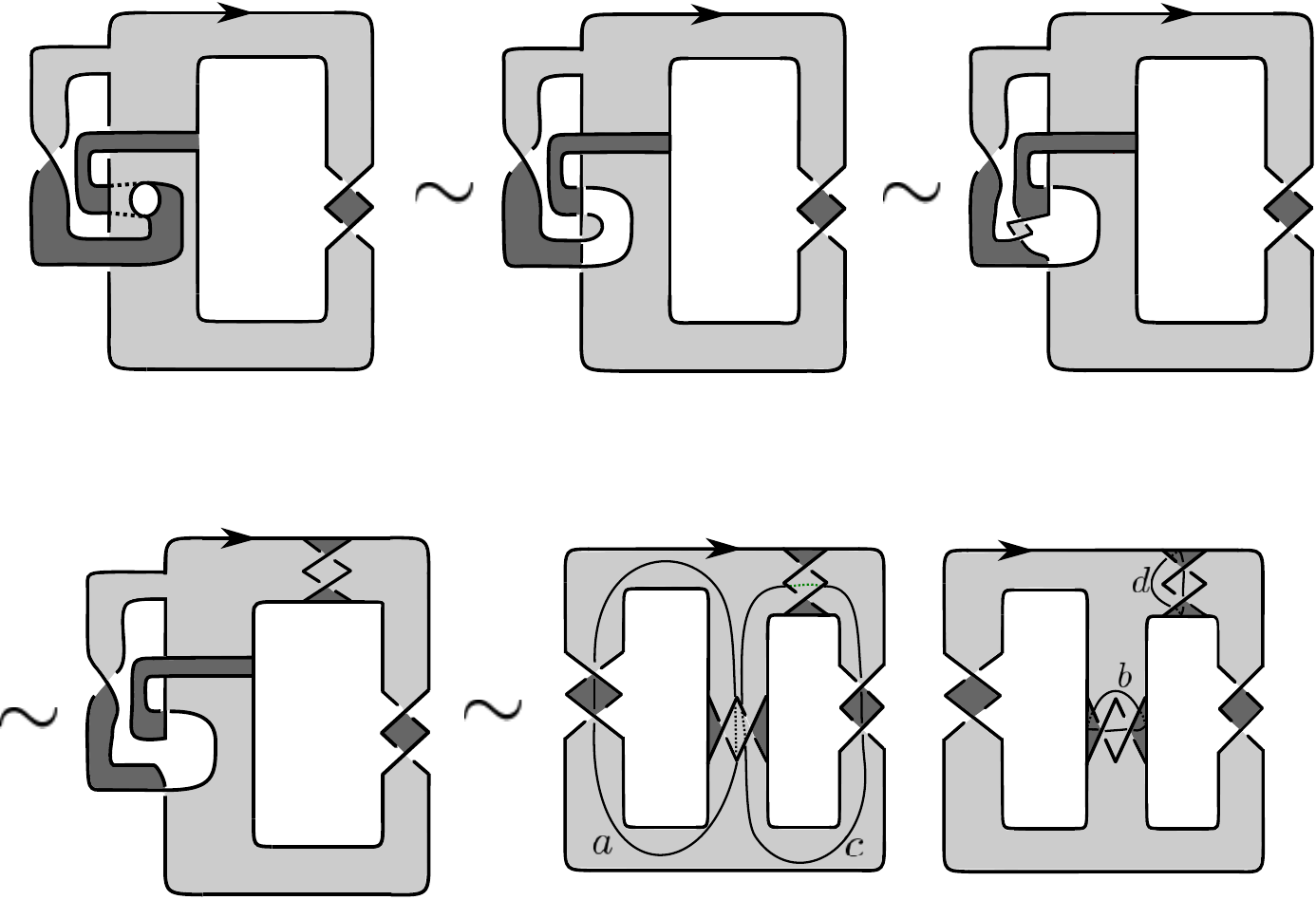}
\caption{The bottom right pictures are fiber surfaces of $6_{3}$ given by a plumbing of some Hopf bands. 
The loops $a$, $b$, $c$ and $d$ are core lines of these Hopf bands. }
\label{figure:monodromy}
\end{figure}
\begin{figure}[h]
\centering
\includegraphics[scale=0.9]{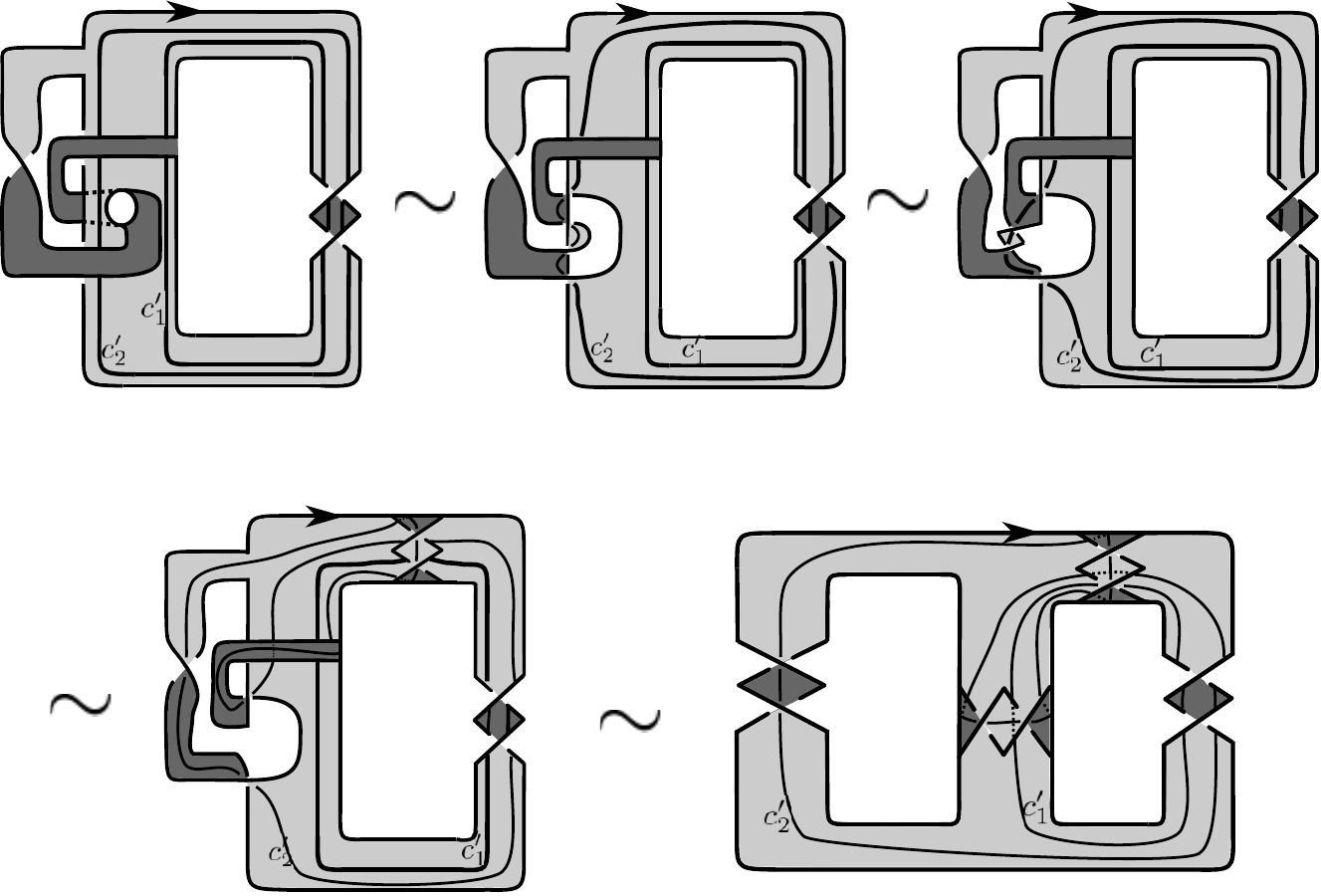}
\caption{The simple closed curves $c'_{1}$ and $c'_{2}$ on the fiber surface of $6_{3}$. }
\label{figure:monodromy2}
\end{figure}
\begin{figure}[h]
\centering
\includegraphics[scale=1.3]{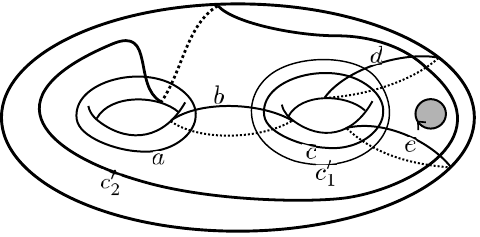}
\caption{The monodromy $f_{n}$ of $A^{n}(6_{3})$ is $t_{c'_{2}}^{n}\circ t_{c'_{1}}^{-n}\circ t^{-1}_{d}\circ t_{b}\circ t^{-1}_{c}\circ t_{a}$. This is equal to $t_{a}^{-1}\circ t_{b}^{-1}\circ t_{e}\circ t_{c}^{n}\circ t_{e}^{-1}\circ t_{b}\circ t_{a}\circ t_{c}^{-n}\circ t^{-1}_{d}\circ t_{b}\circ t^{-1}_{c}\circ t_{a}$, where $e$ is the circle depicted in this picture.}
\label{figure:monodromy3}
\end{figure}
\par
Let $K_{n}$ be the fibered knot $A^{n}(6_{3})$.
Then the closed monodromies $\widehat{f}_{n}$ are conjugate with each other.
It follows from  two facts:
\begin{enumerate}
\item $0$-surgeries on $K_{n}$ are the same $3$-manifold
(which is the surface bundle over $S^{1}$ with monodromy $\widehat{f}_{n}$ and whose first Betti number is one). 
\item The monodromy of any surface bundle over $S^{1}$ with first Betti number one is
 unique up to conjugation. 
\end{enumerate}
\par
Hence, the closed monodromies $\widehat{f}_{n}$ do not distinguish the knots $K_{n}$. 
On the other hand, we see that the monodromies $f_{n}$ distinguish the knots $K_{n}$ by Remark~\ref{rem:plane-field} below. 
\begin{rem}\label{rem:plane-field}
Let $\xi_{n}$ be the contact structure on $S^{3}$ supported by the open book decomposition $(F, f_{n})$.   
Oba told us that 
\[
d_{3}(\xi_{n})=-n^2-n+\frac{3}{2}, 
\]
where $d_{3}$ is the invariant of plane fields given by Gompf \cite{Gompf}. 
In order to compute $d_{3}(\xi_{n})$,
he used the formula for $d_{3}$ introduced in \cite{GHS, EO},
see  \cite{ATagami2} for the details.
By this computation, if $K_{n}$ and $K_{m}$ are the same fibered knots, 
then  $n=m$ or $n+m=-1$. 
Moreover if $n+m=-1$,  we can check that  $K_{n}$ and $K_{m}$ are the same fibered knots.
As a result, we see that $K_{n}$ and $K_{m}$ are the same fibered knots if and only 
if  $n=m$ or $n+m=-1$.
In particular, knots $K_{n}$ $(n\ge 0)$ are mutually distinct.
\par
For a knot $K$ with an annulus presentation $(A,b)$, in general, it is hard to distinguish $A^{n}(K)$ and $A^{m}(K)$. 
Indeed, they have the same Alexander modules. 
Osoinach \cite{Osoinach} and Teragaito \cite{Teragaito} used the hyperbolic structure of the complement of $A^{n}(K)$ to solve the problem (more precisely, they considered the hyperbolic volume of $A^n(K_n)$). 
On the other hand, in Oba's method, we consider contact structures. 
\end{rem}
%
%
%
%
\begin{rem}
In the proof of Theorem~\ref{thm:main}, we proved that $K_{0}\# \overline{K_{1}}$ is not ribbon. 
By the same discussion, if $K_{n}\neq K_{m}$, we also see that $K_{n}\# \overline{K_{m}}$ is not ribbon and it 
is a counterexample for either Conjecture~\ref{conj:AK2} or the slice-ribbon conjecture.
In particular, by Remark~\ref{rem:plane-field}, we obtain infinitely many fibered potential counterexamples to 
the slice-ribbon conjecture by utilizing annulus twists. 
\end{rem}

\noindent {\bf Acknowledgements.} 
The authors would like to thank Lee Rudolph for his comments on tight fibered knots. 
They deeply thank Takahiro Oba for telling them about the result in Remark~\ref{rem:plane-field}. 
They also thank Hidetoshi Masai for pointing out the gap of a proof of Theorem~\ref{thm:main} in the first draft
based on the work of Bonahon \cite{Bonahon}.
They also thank the referee for a careful reading and helpful comments.
The first author was supported by JSPS KAKENHI Grant number 13J05998.  
The second author was supported by JSPS KAKENHI Grant number 15J01087. 

\end{document}